\documentclass[11pt]{amsart} 
\usepackage[english]{babel}
\usepackage{graphicx,epsfig}
\usepackage{parskip}
\usepackage{bbm}
\usepackage{listings}
\usepackage{amsthm}
\usepackage{comment}
\usepackage{enumerate}
\usepackage{esvect}
\usepackage{amsmath,amssymb,latexsym, amsfonts, amscd, amsthm, xy}
\input{xy}
\xyoption{all}

\usepackage{mathrsfs}
\usepackage{appendix}

\usepackage{mathtools}

\usepackage[margin=1in]{geometry}

\usepackage[dvipsnames]{xcolor}

\makeindex 

\usepackage{hyperref}
\hypersetup{pdftoolbar=true, pdftitle={GKform}, pdffitwindow=true, colorlinks=true, citecolor=blue, filecolor=black, linkcolor=purple, urlcolor=red, hypertexnames=false}

\theoremstyle{plain}

\newtheorem{theorem}{Theorem}[section]

\newtheorem{proposition}[theorem]{Proposition}
\newtheorem{corollary}[theorem]{Corollary}

\newtheorem*{assumption*}{Assumption}
\newtheorem{lemma}[theorem]{Lemma}

\theoremstyle{definition}
\newtheorem{definition}[theorem]{Definition}
\newtheorem{remark}[theorem]{Remark}
\newtheorem{example}[theorem]{Example}
\newtheorem*{goal*}{Goal}
\newtheorem*{problem*}{Comment}

\def\lra{{\longrightarrow \, }}
\def\SL{{\rm SL}}
\def\GL{{\rm GL}}

\def\A{\mathbb{A}}

\DeclareMathOperator{\cyc}{cyc}
\DeclareMathOperator{\cond}{cond}

\DeclareMathOperator{\spe}{sp}

\DeclareMathOperator{\dR}{dR} \DeclareMathOperator{\pr}{pr}

\DeclareMathOperator{\Det}{Det}

 \DeclareMathOperator{\cl}{cl}

\DeclareMathOperator{\et}{et}

\DeclareMathOperator{\End}{End} \DeclareMathOperator{\CH}{CH}
\DeclareMathOperator{\AJ}{AJ} 

\DeclareMathOperator{\spec}{Spec} \DeclareMathOperator{\sgn}{sgn}

\DeclareMathOperator{\ord}{ord}

\DeclareMathOperator{\lcm}{lcm}

\DeclareMathOperator{\tr}{tr}

\DeclareMathOperator{\ind}{ind}

\DeclareMathOperator{\gks}{GKS}

\DeclareMathOperator{\orb}{orb}

\def\d{\mathrm{d}}

\def\Z{\mathbb{Z}}

\def\F{\mathbb{F}}

\def\Q{\mathbb{Q}}

\def\C{\mathbb{C}}

\def\R{\mathbb{R}}

\def\bdf{\begin{defn}}
\def\edf{\end{defn}}

\def\Gal{{\rm Gal}}

\def\ab{{\rm ab}}

\def\cC{{\mathcal C}}

\def\ab{\text{ab}}

\def\d1{d^{(1)}}

\def\bA{\mathbb{A}}

\def\d{\mathrm{d}}

\def\T{\mathbb{T}}

\usepackage{tikz}
\usetikzlibrary{positioning} 
\usepackage{tikz-cd}
\usetikzlibrary{arrows,graphs,decorations.pathmorphing,decorations.markings}

\tikzset{
commutative diagrams/.cd,
arrow style=tikz,
diagrams={>=latex}}

\makeatletter
\let\@wraptoccontribs\wraptoccontribs
\makeatother

\begin{document} 

 \title{Twisted triple product root numbers and a cycle of Darmon--Rotger}
\author{David T.-B. G. Lilienfeldt}
\address{Mathematical Institute, Leiden University, The Netherlands} 
\email{d.t.b.g.lilienfeldt@math.leidenuniv.nl}
\date{\today}

\begin{abstract}
We consider an algebraic cycle on the triple product of the prime level modular curve $X_0(p)$ with origins in work of Darmon and Rotger. It is defined over the quadratic extension of $\Q$ ramified only at $p$ whose associated quadratic character $\chi$ is the Legendre symbol at $p$. We prove that it is null-homologous and describe actions of various groups on it.  
For any three normalised cuspidal eigenforms $f_1, f_2, f_3$ of weight $2$ and level $\Gamma_0(p)$, we prove that the global root number of the twisted triple product $L$-function $L(f_1\otimes f_2\otimes f_3\otimes \chi, s)$ is $-1$. Assuming conjectures of Beilinson and Bloch, and guided by the Gross--Zagier philosophy, this suggests that the Darmon--Rotger cycle could be non-torsion, although we do not currently have a proof of this.
\end{abstract}

\maketitle

\section{Introduction}

\subsection{The Darmon--Rotger cycle}\label{s:11}

Let $p$ be a prime number and fix a $p$-th root of unity $\zeta_p$ in $\bar{\Q}$. Let $X_0(p)$ denote the proper modular curve of level $\Gamma_0(p)$ defined over $\Q$ (see Section \ref{s:mod}). Let $K\subset \Q(\zeta_p)$ denote the unique quadratic extension of $\Q$ ramified only at $p$, and write $\Gal(K/\Q)=\{1,\tau\}$. 
Given an elliptic curve $E$ over a $\Q(\zeta_p)$-scheme $S$ and three pairwise distinct finite flat cyclic subgroup schemes $C_1, C_2, C_3 \subset E[p]$ of rank $p$ with generators $P_1, P_2, P_3\in E[p](S)$, denote by $e_p$ the Weil pairing \cite[(2.8.5.1)]{katzmazur} on $E[p]$ and write 
\[ e_p(P_2,P_3)=\zeta_p^a, \qquad e_p(P_3,P_1)=\zeta_p^b, \qquad e_p(P_1,P_2)=\zeta_p^c, \qquad \text{ with } a,b,c \in \mathbb{F}_p^{\times}. \]
Using this notation, define the invariant
\[
o(E; C_1, C_2, C_3) := abc \in \mathbb{F}_p^{\times}\slash (\mathbb{F}_p^{\times})^{(2)},
\]
where $(\mathbb{F}_p^{\times})^{(2)}$ denotes the group of non-zero quadratic residues modulo $p$. This invariant does not depend on the choices of generators of the cyclic subgroups. Darmon--Rotger \cite{darmonrotgerast} defined two curves $\Delta_+, \Delta_-\subset X_0(p)^3_{\bar{\Q}}$ respectively as the schematic closures of
\[
\{(E,C_1), (E,C_2), (E,C_3) \colon o(E; C_1, C_2, C_3)=1 \}
\]
and 
\[
\{(E,C_1), (E,C_2), (E,C_3) \colon o(E; C_1, C_2, C_3)\neq 1\}.
\] 
We refer to the difference of these two curves as the Darmon--Rotger cycle and prove the following:
\begin{theorem}\label{thm:intro2}
    The Darmon--Rotger cycle is null-homologous and its rational equivalence class gives rise to an element $\Xi:=\Delta_+-\Delta_-$ in the Chow group $\CH^2(X_0(p)^3_{K})_0^{\tau=-1}.$ The element $\Xi$ is fixed by the action of the symmetric group $S_3$ if $p\equiv 1\pmod 4$. If $p\equiv 3\pmod 4$, then $S_3$ acts on $\Xi$ by the sign character.
\end{theorem}

In order to prove this, we give an alternative description of $\Delta_+$ and $\Delta_-$ as images of certain maps $X(p)\lra X_0(p)^3$, where $X(p)$ denotes the proper modular curve of full level $p$ structure defined over $\Q(\zeta_p)$ (see Section \ref{s:mod}).

\subsection{Root numbers}

Consider $F:=f_1\otimes f_2\otimes f_3 \in S_2(\Gamma_0(p))^{\otimes 3}$, the tensor product of three normalised cuspidal eigenforms with respect to the Hecke algebra $\T=\End_{\Q}(J_0(p))\otimes \Q$. Let $\chi$ denote the Legendre symbol at $p$, which is the quadratic character associated to the extension $K/\Q$. We consider the triple product $L$-function $L(F, \chi, s)$ attached to the compatible family of $8$-dimensional $\ell$-adic $\Gal(\bar{\Q}/\Q)$-representations $$V_{\ell}(F,\chi):=V_{\ell}(f_1)\otimes V_{\ell}(f_2)\otimes V_{\ell}(f_3)\otimes \chi,$$ where $V_{\ell}(f_i)$, for $i=1,2,3$, is the usual $2$-dimensional representation attached to a cuspform of weight $2$. Define the completed $L$-function $$\Lambda^*(F,\chi,s):=2^4 p^{4s}(2\pi)^{3-4s}\Gamma(s-1)^3\Gamma(s)L(F, \chi, s).$$ This function has analytic continuation to the entire $s$-plane and satisfies the functional equation $$\Lambda^*(F,\chi,s)=W(F,\chi)\Lambda^*(F,\chi,4-s),$$ where $W(F,\chi)\in \{ \pm 1 \}$ is the global root number \cite{harriskudla,PSR}. We prove the following:
\begin{theorem}\label{thm:intro}
    The $L$-function $L(F,\chi,s)$ vanishes to odd order at its center $s=2$.
\end{theorem} 

The proof consists in showing that $W(F,\chi)=-1$, which boils down to computing the  epsilon factor of the $8$-dimensional Weil--Deligne representation of $F\otimes \chi$ at $p$. The main difference with the untwisted case lies in the fact that the epsilon factor of this Weil--Deligne representation at $p$ is equal to the epsilon factor of the Weil representation of $F\otimes \chi$ at $p$ (see \eqref{epsss} and Remark \ref{rem:end}).

Triple product $L$-functions for three newforms of weight $2$ and same square-free level were studied in detail by Gross--Kudla \cite{grosskudla}. In particular, they gave a formula for the global root number. In the case of prime level $p$, the global root number of $L(F,s)$ is $W(F)=a_p(f_1)a_p(f_2)a_p(f_3)$, where $a_p(f_i)$ denotes the $p$-th Fourier coefficient of $f_i$ at the cusp $\infty$. 
The twisted case considered in this note corresponds to the triple product of three newforms of levels $p, p, p^2$ (see Remark \ref{rem:psq}), and thus falls outside the scope of \cite{grosskudla}.

In the case of a single normalised cuspidal eigenform $f\in S_2(\Gamma_0(p))$, the global root numbers are $W(f)=a_p(f)$ and $W(f, \chi)=-\chi(-1)$ \cite{pacetti, rohrlich94}.

\subsection{The Beilinson--Bloch--Kato conjectures}\label{s:BBK}
Let $Y$ be a smooth and proper variety over a number field $F$. Taken together, the Beilinson--Bloch and Bloch--Kato conjectures \cite{bloch, blochkato} predict that for each integer $i \geq 0$ and prime $\ell$, the $\ell$-adic \'etale Abel--Jacobi map \begin{equation}\label{eq:AJ}
    \AJ_{\ell}^{i} \colon \CH^{i}(Y)_0 \otimes \Q_\ell \longrightarrow H^1_f(F, H_{\et}^{2i-1}(Y_{\bar{F}},\Q_\ell(i)))
    \end{equation}
from the null-homologous Chow group to the Bloch--Kato Selmer group is an isomorphism of $\Q_\ell$-vector spaces and moreover these spaces have dimension $\ord_{s = i} L(H^{2i-1}(Y),s)$,
the $L$-function being the one attached to the compatible system of $\ell$-adic $\Gal(\bar{F}/F)$-representations $H^{2i-1}_{\et}(Y_{\bar{F}}, \Q_\ell
)$. The Beilinson--Bloch--Kato conjectures are compatible with algebraic correspondences.

In the setting of this note, we expect from the Beilinson--Bloch--Kato conjectures that
\begin{equation}\label{BBK}
    \ord_{s = 2}L(F,\chi,s) = \dim_{K_F} t_F \CH^2(X_0(p)^3_K)_0^{\tau=-1}=\dim_{K_{F,\ell}} H^1_f(\Q, V_{F,\chi,\ell}(2)),
\end{equation}
where $K_{F,\ell}$ is the completion of the Hecke field $K_F$ of $F$ at a fixed prime above $\ell$, and the correspondence $t_F\in \CH^3(X_0(p)^6)_{K_F}$ is some choice of $K_F$-linear combination of
tensor products of Hecke correspondences projecting to the $1$-dimensional $F$-isotypic
component of the $(\T^{\otimes 3} \otimes \R)$-module $H^0(X_0(p)^3 , \Omega^3_{X_0(p)^3}) \otimes \R = H^0(X_0(p), \Omega^1_{X_0(p)})^{\otimes 3} \otimes \R$. Theorem \ref{thm:intro} implies that
\begin{equation}\label{coro1}
    \ord_{s = 2}L(F,\chi,s) \geq 1.
\end{equation}
In view of the Beilinson--Bloch--Kato conjectures, we thus expect to have the lower bound 
\begin{equation}\label{eqint}
    \dim_{K_F} t_F \CH^2(X_0(p)^3_K)_0^{\tau=-1}\geq 1.
\end{equation}

A natural question arises from \eqref{eqint} and Theorem \ref{thm:intro2}: when are $t_F\Xi$ and $\AJ_{\ell}^2(t_F\Xi)$ non-trivial? It would be interesting to know the answer to this, especially given the canonical nature of the Darmon--Rotger cycle (it does not depend on a choice of base-point nor on a projector to make it null-homologous, as opposed to the modified diagonal cycle in Remark \ref{rem:gks} below). At the moment, we have no answers to offer, only speculations.

In analogy with the situation for Heegner points \cite{GZ}, the Gross--Zagier philosophy might lead us to expect that 
\begin{equation}\label{gz}
t_F \Xi \neq 0 \iff \ord_{s=2} L(F,\chi,s)=1.
\end{equation}
In particular, the existence of one triple product of normalised cuspidal eigenforms such that $L'(F,\chi,2)\neq 0$ would be enough to imply that $\Xi$ has infinite order in $\CH^2(X_0(p)^3_K)_0$.

\begin{remark}\label{rem:gks}
Over $\Q$, a naturally occurring cycle is the modified diagonal cycle 
\[
\Delta_{\gks}(e):=P_{\gks}(e)(\Delta)\in \CH^2(X_0(p)^3)_0
\]
first considered in work of Gross--Kudla \cite{grosskudla} and Gross--Schoen \cite{grsc}. Here $e\in X_0(p)(\Q)$ is a fixed base-point and $P_{\gks}(e)$ is the Gross--Schoen projector (Definition \ref{def:pgks}) whose effect is to make the small diagonal $\Delta\subset X_0(p)^3$ null-homologous. When $\ord_{s=2} L(F,s)$ is even, $\AJ^2_{\ell}(t_F\Delta_{\gks}(e))$ is trivial for all $\ell$ and base-point $e$ \cite{lilienfeldtdiagonal} (hence $t_F\Delta_{\gks}(e)$ is trivial assuming injectivity of \eqref{eq:AJ}). When $\ord_{s=2} L(F,s)=1$, $t_F\Delta_{\gks}(\infty)$ is non-trivial by the height formula conjectured in \cite{grosskudla} and announced in \cite{yzz}. If in addition we have $\ord_{s=2} L(F,\chi,s)=1$, then  $\ord_{s=2} L(F/K,s)=2$, and it is conceivable in view of \eqref{gz} that
\[
t_F \CH^2(X_0(p)^3_K)_0= K_F \cdot t_F\Delta_{\gks}(\infty)\oplus K_F\cdot t_F \Xi,
\]
although this is pure speculation at the moment. 
\end{remark}

\subsection{Composite levels and higher weights}

Let $M\geq 1$ be an integer that is not divisible by $p$. Let $X(M,p)$ denote the proper modular curve over $\Q$ of level $\Gamma_1(M)\cap \Gamma_0(p)$ classifying isomorphism classes of pairs $(E,C)$ consisting of a generalized elliptic curve $E$ with $\Gamma_1(M)$-level structure and a cyclic subgroup $C$ of level $p$ (see Section \ref{s:mod}). Then one can define cycles $\Delta^{+}_{01}$ and $\Delta^-_{01}$ in $\CH^2(X(M,p)_K^3)$ by mimicking the definitions of the cycles $\Delta_+$ and $\Delta_-$ given in Section \ref{s:11}. This is in fact the level of generality in which these cycles appear in \cite[Lemma 4.1]{darmonrotgerast} (from where the notation $\Delta_{01}^{\pm}$ is borrowed). It seems possible to generalize Theorem \ref{thm:intro2} to this more general situation. Given normalized newforms $f_i \in S_2(\Gamma_0(M_i p))$ with $p\nmid M_i$ for $i=1,2,3$ and $M=\lcm(M_1,M_2,M_3)$, it appears likely that the techniques used in the proof of Theorem \ref{thm:intro} can be adapted to yield the local root number equality $W_p(f_1\otimes f_2\otimes f_3\otimes \chi)=1$. We have chosen to focus on the simplest case of prime level in this note, for which the global root number statement of Theorem \ref{thm:intro} is particularly enticing for the discussion in Section \ref{s:BBK}.

Let $\mathbf{f}_1, \mathbf{f}_2, \mathbf{f}_3$ be three $p$-adic Hida families of tame levels $M_1, M_2, M_3$ subject to the assumptions of \cite[p. 57]{darmonrotgerast}. As above, take $M=\lcm(M_1,M_2,M_3)$. Let $(x,y,z)$ be a crystalline point of the associated weight space of balanced weight $(k,\ell,m)=(k_0+2, \ell_0+2, m_0+2)$ with $k_0,\ell_0,m_0\geq 0$ such that $k_0+\ell_0+m_0=2r$ is even. Darmon and Rotger \cite[(2.27)]{darmonrotgerast} have constructed a {\it generalized Kato class} $\kappa:=\kappa_{\infty}(1,1,1;1)$ that is a three variable family of cohomology classes parametrized by points in weight space and taking values in the three-parameter family of self-dual Tate twists of the Galois representations attached to the different specializations of a triple of Hida families. When $(k_0, \ell_0, m_0)=(0,0,0)$, the specialization at $(x,y,z)$ of $\kappa$ interpolates the $p$-adic \'etale Abel--Jacobi images of the modified diagonal cycles of \cite[(2.14)]{darmonrotgerast}. When $(k_0, \ell_0, m_0)>(0,0,0)$, Darmon and Rotger \cite{DR1} have constructed generalized diagonal cycles $\Delta^{k_0,\ell_0,m_0}\in \CH^{r+2}(\mathcal{A}^{k_0}\times \mathcal{A}^{\ell_0}\times \mathcal{A}^{m_0})$ \cite[p. 62, l. 22]{darmonrotgerast}, where $\mathcal{A}^{n}$ denotes the $(n+1)$-dimensional Kuga--Sato variety fibered over $X_1(M)$ for any $n\geq 1$. After applying a projector to make them null-homologous, these cycles give rise to the {\it generalized Gross--Kudla--Schoen cycles} $\Delta_0^{k_0,\ell_0,m_0}\in \CH^{r+2}(\mathcal{A}^{k_0}\times \mathcal{A}^{\ell_0}\times \mathcal{A}^{m_0})_0$ of \cite[Definition 3.3]{DR1}. The $(\mathbf{f}_{1,x}, \mathbf{f}_{2,y}, \mathbf{f}_{3,z})$-isotypic component of the $(x,y,z)$-specialization of $\kappa$ is then the $p$-adic \'etale Abel--Jacobi image of $\Delta_0^{k_0,\ell_0,m_0}$ \cite[Theorem 4.1]{darmonrotgerast}. This identification enables the authors to link the global cohomology class $\kappa$ with an (unbalanced) Garrett--Hida triple product $p$-adic $L$-function associated to $\mathbf{f}_1, \mathbf{f}_2, \mathbf{f}_3$ via a reciprocity law \cite[Theorem 5.12]{darmonrotgerast} using the main theorem of \cite{DR1}. This reciprocity law is a key ingredient in the proof of the main result of \cite{darmonrotgerast1}, which pertains to the equivariant Birch and Swinnerton-Dyer conjecture at special points of weight $(2,1,1)$.  

As an intermediate step in their proof of \cite[Theorem 4.1]{darmonrotgerast}, Darmon and Rotger are led to consider a cycle $\Delta_{01}^{k_0,\ell_0,m_0} \in \CH^{r+2}(\mathcal{A}(M,p)^{k_0}\times \mathcal{A}(M,p)^{\ell_0}\times \mathcal{A}(M,p)^{m_0})$ \cite[p. 62, l. 25]{darmonrotgerast} fibered above the cycle $\Delta^+_{01}+\Delta^-_{01}\in \CH^2(X(M,p)^3)$. (Here, $\mathcal{A}(M,p)^n$ denotes the $(n+1)$-dimensional Kuga--Sato variety fibered over $X(M,p)$.) It seems possible, although this is not done in \cite{darmonrotgerast}, to define cycles $\Delta_{01, +}^{k_0,\ell_0,m_0}, \Delta_{01, -}^{k_0,\ell_0,m_0} \in \CH^{r+2}((\mathcal{A}(M,p)^{k_0}\times \mathcal{A}(M,p)^{\ell_0}\times \mathcal{A}(M,p)^{m_0})_K)$ fibered over $\Delta^{+}_{01}$ and $\Delta^{-}_{01}$ respectively and such that $\Delta_{01}^{k_0,\ell_0,m_0}=\Delta_{01, +}^{k_0,\ell_0,m_0}+\Delta_{01, -}^{k_0,\ell_0,m_0}$. One could then envision a notion of {\it generalized Darmon--Rotger cycle} defined as the difference $\Delta_{01, +}^{k_0,\ell_0,m_0}-\Delta_{01, -}^{k_0,\ell_0,m_0}$. For the future, it would be worthwhile to study these cycles in depth in the spirit of \cite{DR1,DR2}. In particular, the $p$-adic syntomic Abel--Jacobi images of such cycles could turn out to encode significant arithmetic properties; the ambient varieties have bad reduction at $p$ in this case, so such Abel--Jacobi calculations would be challenging and of independent interest as they appear to lie beyond the scope of the techniques developed in \cite{BLZ} and employed in \cite{DR1,DR2}.  

\subsection{Conventions}\label{not:emb}
By default, all Chow groups are with $\Q$-coefficients.
All number fields are viewed as embedded in a fixed algebraic closure $\bar{\Q}$ of $\Q$. Moreover, we fix a complex embedding $\bar{\Q}\hookrightarrow \C$, as well as $p$-adic embeddings $\bar{\Q}\hookrightarrow \C_p$ for each rational prime $p$. 

\subsection{Outline}
Section \ref{s:construction} describes the Darmon--Rotger cycle and contains the proof of Theorem \ref{thm:intro2}. Section \ref{s:WD} provides the necessary background on Weil--Deligne representations and local factors needed for the proof of Theorem \ref{thm:intro}, which is given in Section \ref{s:trip}.

\section{Triple product cycles}\label{s:construction}

Let $p\geq 5$ be a prime number. 
We give a description of the Darmon--Rotger cycle using certain maps $X(p)\lra X_1(p)^3\lra X_0(p)^3$ between various modular curves. We use this description to prove Theorem \ref{thm:intro2}. 

\subsection{Modular curves}\label{s:mod}

We begin by defining the relevant modular curves; see \cite{delignerap,diamondim,katzmazur} for details.

In the notation of \cite[\S 1.4]{katz}, let $\bar{M}_p$ denote the fine moduli scheme representing isomorphisms classes of pairs $(E, \alpha_p)$ consisting of a generalized elliptic curve $E$ over a $\Q$-scheme $S$ together with a full level $p$ structure \cite[(3.1)]{katzmazur}, i.e., an isomorphism of $S$-group schemes $\alpha_p : (\Z/p\Z)^2_S \overset{\sim}{\lra} E[p]$. We often write $\alpha_p=(P,Q)$ where $P=\alpha_p((1,0))\in E[p](S)$ and $Q=\alpha_p((0,1))\in E[p](S)$. The scheme $\bar{M}_p$ is a smooth proper curve over $\Q$, whose base change to $\Q(\zeta_p)$ is the disjoint union of $p-1$ geometrically connected smooth proper curves $X^j(p)$ with $j\in \{ 1, \ldots, p-1 \}$. The curve $X^j(p)$ classifies pairs $(E, (P, Q))$ consisting of a generalized elliptic curve $E$ over a $\Q(\zeta_p)$-scheme $S$ and a basis $(P, Q)$ of $E[p](S)$ satisfying $e_p(P, Q)=\zeta_p^j$, where $e_p$ denotes the Weil pairing. We shall focus on the first component $X(p):=X^1(p)$ (because, as we will see in Section \ref{s:gal}, considering the other components does not yield additional cycles).

Let $X_1(p)$ denote the modular curve over $\Q$ of level $\Gamma_1(p)$, which is the fine moduli scheme representing isomorphism classes of pairs $(E,P)$ consisting of a generalized elliptic curve over a $\Q$-scheme $S$ together with a point $P\in E^{\mathrm{reg}}(S)$ of exact order $p$ with the property that for all geometric points $s\colon \spec(k)\lra S$ the image of the resulting immersion $(\Z/p\Z)_k \hookrightarrow E_k^{\mathrm{reg}}$ meets every component \cite[\S 9.3]{diamondim}. It is a geometrically connected smooth proper curve over $\Q$.

Any element $d\in (\Z/p\Z)^\times$ gives rise to an automorphism of $X_1(N)$ defined over $\Q$ via the diamond operator $\langle d\rangle$ mapping $(E,P)$ to $(E,dP)$. The quotient $X_0(p):=X_1(p)/(\Z/p\Z)^\times$ of $X_1(p)$ by this action is the modular curve over $\Q$ of level $\Gamma_0(p)$. It is a coarse moduli scheme representing isomorphisms classes of pairs $(E,C)$ consisting of a generalized elliptic curve over a $\Q$-scheme $S$ together with a finite flat subgroup scheme $C\subset E^{\mathrm{reg}}$ with the property that for all geometric points $s\colon \spec(k)\lra S$ the fiber $C_k$ is a cyclic group of order $p$ whose image $C_k \hookrightarrow E_k^{\mathrm{reg}}$ meets every component \cite[\S 8.2]{diamondim}. The curve $X_0(p)$ is a geometrically connected smooth proper curve over $\Q$. In the rest of the paper, we will assume that $X_0(p)$ has positive genus, i.e., $p\not\in \{ 2,3,5,7,13\}$.

\subsection{Cycles on $X_1(p)^3$}\label{s2.1}

Let $x_i=(a_i, b_i)\in \F_p^2\setminus \{(0, 0)\}$ with $i\in \{1, 2, 3 \}$ and consider the map 
\[
\widetilde{\varphi}_{(x_1, x_2, x_3)} \colon \bar{M}_p \lra X_1(p)^3
\]
defined over $\Q$ by
\begin{equation*}
 (E, (P, Q))\mapsto ((E, a_1P+b_1Q), (E, a_2P+b_2Q), (E, a_3P+b_3Q)).
\end{equation*}
Denote by 
\[
\widetilde{\Delta}_{(x_1, x_2, x_3)}:=\widetilde{\varphi}_{(x_1, x_2, x_3)}(X(p))\in \CH^2(X_1(p)_{\Q(\zeta_p)}^3)
\]
the image of $X(p)$ under this map. 
We have a collection 
\[
\widetilde{\mathcal{C}}:=\left\{ \widetilde{\Delta}_{(x_1, x_2, x_3)} : (x_1, x_2, x_3)\in (\F_p^2\setminus \{(0, 0)\})^3 \right\} \subset \CH^2(X_1(p)_{\Q(\zeta_p)}^3)
\]
that inherits from $X(p)$ and $X_1(p)^3$ various actions of groups, which we will now define and study.

\subsubsection{Action of the group $\SL_2(\F_p)$}

There is a natural left action of the group $\SL_2(\F_p)$ on $X(p)$, as can be seen, using the moduli interpretation, as follows: if $\left(\begin{smallmatrix} \alpha & \beta \\ \gamma & \delta \end{smallmatrix}\right)\in \SL_2(\F_p)$, then 
\[
\left(\begin{matrix} \alpha & \beta \\ \gamma & \delta \end{matrix}\right) \cdot (E, (P, Q)):=(E, (\alpha P + \beta Q, \gamma P+\delta Q)). 
\]
Because the determinant is one, the Weil pairing on the basis is preserved. The above action naturally induces a right action of $\SL_2(\F_p)$ on the set $\widetilde{\cC}$ via 
\[
\widetilde{\Delta}_{x_1,x_2, x_3}\cdot \kappa := \widetilde{\varphi}_{(x_1, x_2, x_3)}\circ \kappa (X(p)),
\]
but since $\SL_2(\F_p)$ acts by automorphisms this action is the trivial one. A quick calculation reveals that 
\[ 
\widetilde{\Delta}_{(x_1, x_2, x_3)}\cdot \kappa=\widetilde{\Delta}_{(x_1, x_2, x_3)\cdot \kappa}
\] 
where the right action of $\SL_2(\F_p)$ on the set $(\F_p^2\setminus \{(0, 0)\})^3$ is defined as follows. Let $\kappa=\left(\begin{smallmatrix} \alpha & \beta \\ \gamma & \delta \end{smallmatrix}\right)\in \SL_2(\F_p)$ and $(x_1, x_2, x_3)\in (\F_p^2\setminus \{(0, 0)\})^3$ with $x_i=(a_i, b_i)$, $i=1,2,3$, then write the vector $(x_1, x_2, x_3)$ as a $3\times 2$ matrix and multiply on the right by $\kappa$: 
\begin{align*}
(x_1, x_2, x_3)\cdot \kappa : & = 
\left(\begin{matrix} a_1 & b_1 \\ a_2 & b_2 \\ a_3 & b_3 \end{matrix}\right)
\left(\begin{matrix} \alpha & \beta \\ \gamma & \delta \end{matrix}\right) \\
& =((a_1\alpha + b_1\gamma, a_1 \beta+b_1\delta), (a_2\alpha + b_2\gamma, a_2 \beta+b_2\delta), (a_3\alpha + b_3\gamma, a_3 \beta+b_3\delta)).
\end{align*}

It follows that the indexing set of the cycles can be taken to be 
\[ 
\widetilde{I}:=(\F_p^2\setminus \{(0, 0)\})^3/\SL_2(\F_p).
\] 
We shall write $[x_1, x_2, x_3]$ for the image of $(x_1, x_2, x_3)$ in $\widetilde{I}$. 
To understand the set $\widetilde{I}$ we introduce a determinant map 
\[ 
\Det : \widetilde{I} \lra (\F_p)^3
\] defined as follows. If $(x_1, x_2, x_3)$ is a representative of a class in $\widetilde{I}$ with $x_i=(a_i, b_i)$ for $i\in \{1,2,3\}$, then 
\[
\Det([x_1, x_2, x_3]):=\left(  
\left\vert\begin{matrix} a_2 & b_2 \\ a_3 & b_3 \end{matrix}\right\vert,
\left\vert\begin{matrix} a_3 & b_3 \\ a_1 & b_1 \end{matrix}\right\vert,
\left\vert\begin{matrix} a_1 & b_1 \\ a_2 & b_2 \end{matrix}\right\vert
\right).
\]
This map is well-defined as follows from the definition of the action of $\SL_2(\F_p)$. 
It is a bijection from the subset $\widetilde{I}^\times:=\Det^{-1}((\F_p^\times)^3)\subset \widetilde{I}$ to $(\F_p^\times)^3$. If $[x_1, x_2, x_3]\in \widetilde{I}^\times$ with $\Det([x_1, x_2, x_3])=(a, b, c)$, then we write $\widetilde{\Delta}_{a, b, c}:=\widetilde{\Delta}_{(x_1, x_2, x_3)}$. We restrict our attention to the subcollection 
\[
\widetilde{\mathcal{C}}^\times:=\left\{ \widetilde{\Delta}_{a,b,c} : (a,b,c)\in (\F_p^\times)^3 \right\} \subset \CH^2(X_1(p)_{\Q(\zeta_p)}^3).
\]

\subsubsection{Action of the diamond operators}

The modular curve $X_1(p)$ carries a natural left action of the group $\F_p^\times$ defined over $\Q$ via the so-called diamond operators. If $d\in \F_p^\times$, then in terms of the modular description we have 
$\langle d\rangle \cdot (E, P)=(E, dP).$
We get an induced action of $(\F_p^\times)^3$ on the triple product $X_1(p)^3$ described by 
\[
\langle d_1, d_2, d_3\rangle\cdot ((E_1, P_1), (E_2, P_2), (E_3, P_3))=((E_1, d_1P_1), (E_2, d_2P_2), (E_3, d_3P_3)). 
\]
This in turn induces a left action of $(\F_p^\times)^3$ on the collection of cycles $\widetilde{\mathcal{C}}$ via 
\[
\langle d_1, d_2, d_3\rangle \cdot \widetilde{\Delta}_{(x_1, x_2, x_3)}:= \langle d_1, d_2, d_3\rangle\circ \widetilde{\varphi}_{(x_1, x_2, x_3)}(X(p)),
\] 
and this action preserves the subcollection $\widetilde{\mathcal{C}}^\times$ since
\[
\langle d_1, d_2, d_3\rangle \cdot \widetilde{\Delta}_{a, b, c}=\widetilde{\Delta}_{d_2d_3a, d_1d_3b, d_1d_2c}. 
\]
The following lemma follows easily from this formula.
\begin{lemma}\label{coro3}
We have
\[
\orb_{\diamond}(\widetilde{\Delta}_{1,1,1})=\left\{ \widetilde{\Delta}_{a,b,c} \: \vert \: a, b, c \in \F_p^\times, \: abc\in (\F_p^\times)^{(2)} \right\}.
\]
Here $(\F_p^\times)^{(2)}$ denotes the set of quadratic residues modulo $p$ and thus the orbit of $\widetilde{\Delta}_{1,1,1}$ has size $\frac{(p-1)^3}{2}$. The stabiliser of $\widetilde{\Delta}_{1,1,1}$ for this action is given by $\{ \langle 1, 1, 1\rangle, \langle-1, -1, -1\rangle \}$.
 As a consequence, there are $2$ orbits for the action of the diamond operators on $\widetilde{\mathcal{C}}^\times$:  
 \[
 \widetilde{\mathcal{C}}^\times=\orb_{\diamond}(\widetilde{\Delta}_{1,1,1})\sqcup \orb_{\diamond}(\widetilde{\Delta}_{1,1,a}),
 \]
 where $a\in \F_p^\times$ is a choice of a non-quadratic residue modulo $p$.
\end{lemma}

\subsubsection{Action of the Galois group $\Gal(\Q(\zeta_p)/\Q)$}\label{s:gal}

We identify $\Gal(\Q(\zeta_p)/\Q)$ with $\F_p^\times$ so that the element of the Galois group $\sigma_i$ indexed by $i\in \F_p^\times$ raises $\zeta_p$ to the $i$-th power. 
Recall that the curve $\bar{M}_p$ is defined over $\Q$. When base-changed to $\Q(\zeta_p)$, the Galois group of $\Q(\zeta_p)$ permutes the $p-1$ connected components $X^j(p)$ of this curve transitively. Using this, $\Gal(\Q(\zeta_p)/\Q)$ acts on $\widetilde{\cC}$ via 
\[
\widetilde{\Delta}_{(x_1,x_2,x_3)}^{\sigma_i} := \widetilde{\varphi}_{(x_1, x_2, x_3)}(\sigma_i (X(p)))=\widetilde{\varphi}_{(x_1, x_2, x_3)}(X^i(p)).
\]
It is then not difficult to prove the following:
\begin{lemma}\label{galaction}
For all $i\in \F_p^\times$ and $(a,b,c)\in (\F_p^\times)^3$, we have
$
\widetilde{\Delta}_{a, b, c}^{\sigma_i}=\widetilde{\Delta}_{ia, ib, ic}.
$
\end{lemma}

\subsection{Cycles on $X_0(p)^3$}\label{S:X0}

There is a natural degree $(p-1)/2$ covering of curves $\pi : X_1(p)\lra X_0(p)$ defined over $\Q$ given by mapping $(E, P)$ to $(E, \langle P\rangle)$. It gives rise to a map on triple products $\pi^3 : X_1(p)^3\lra X_0(p)^3$. Define, for $(x_1, x_2, x_3)\in ((\F_p\times \F_p)\setminus \{ (0, 0) \})^3$, the map 
\[
\varphi_{(x_1, x_2, x_3)}:=\pi^3 \circ \widetilde{\varphi}_{(x_1, x_2, x_3)} : \bar{M}_p \lra X_0(p)^3,
\]
as well as the cycle
\[ 
\Delta_{(x_1, x_2, x_3)}:=\varphi_{(x_1, x_2, x_3)}(X(p))\in \CH^2(X_0(p)_{\Q(\zeta_p)}^3).
\]  
The cycles $\Delta_{(x_1, x_2, x_3)}$ are invariant under the action of the diamond operators on the triples $(x_1, x_2, x_3)$. By taking images under $\pi^3$ of the cycles in $\widetilde{\mathcal{C}}^\times$ indexed by the set $\widetilde{I}^{\times}$, we obtain a collection of cycles $\mathcal{C}^\times$ in $\CH^2(X_0(p)_{\Q(\zeta_p)}^3)$ indexed by the double coset space $(\F_p^\times)^3 \backslash \tilde{I}^\times$, which has cardinality $2$ by Lemma \ref{coro3}. The $2$ cycles are the schematic closures of: 
\begin{itemize}
\item $\Delta_+ :=\pi^3(\widetilde{\Delta}_{1, 1, 1})=\{ ((E, \langle P\rangle), (E, \langle Q\rangle), (E, \langle P+Q\rangle)) \}$
\item $\Delta_- :=\pi^3(\widetilde{\Delta}_{1, 1, a})=\{ ((E, \langle P\rangle), (E, \langle Q\rangle), (E, \langle aP+Q\rangle)) \}$ ($a$ is a non-quadratic residue).
\end{itemize} 
These cycles first appeared in work of Darmon--Rotger \cite[p. 30]{darmonrotgerast}.
It is clear that $\Delta_+$ and $\Delta_-$ match the descriptions given in the introduction.

\subsubsection{Field of definition}

\begin{lemma}\label{lem:fodK}
The cycles $\Delta_+$ and $\Delta_-$ are defined over the quadratic field 
\[ 
K:=\Q\left(\sqrt{\chi(-1)p}\right)\subset \Q(\zeta_p),
\] 
where $\chi$ denotes the Legendre symbol modulo $p$. The non-trivial element $\tau$ of $\Gal(K/\Q)$ interchanges $\Delta_+$ and $\Delta_-$. 
\end{lemma}

\begin{proof}
Let $G(\chi)$ denote the Gauss sum associated to $\chi$ given by the expression
\[
G(\chi):=\sum_{n=0}^{p-1} \zeta_p^{n^2}.
\]
The equality $G(\chi)^2=\chi(-1)p$ goes back to Gauss and implies that $K$ is the quadratic subfield of the cyclotomic field $\Q(\zeta_p)$. We have 
\[
\sigma_i(G(\chi))=\sum_{n=0}^{p-1}\zeta_p^{in^2}=G(\chi) \iff i\in (\F_p^{\times})^{(2)}, 
\]
and as a consequence $\Gal(\Q(\zeta_p)/K)\simeq (\F_p^{\times})^{(2)}$ and $\Gal(K/\Q)\simeq \F_p^{\times}/(\F_p^{\times})^{(2)}$. Thus $\tau$ acts as $\sigma_a$ where $a\in \F_p^\times$ is not a square. It follows from Lemmas \ref{galaction} and \ref{coro3} that both cycles in $\cC^\times$ are fixed by $\Gal(\Q(\zeta_p)/K)$ and moreover that 
\[
\Delta_+^{\tau}=\pi^3(\widetilde{\Delta}_{1,1,1}^{\tau})=\pi^3(\widetilde{\Delta}_{a,a,a})=\pi^3(\widetilde{\Delta}_{1,1,a})=\Delta_-. 
\]
\end{proof}

\subsubsection{Action of the symmetric group $S_3$}

The symmetric group $S_3$ acts on $X_0(p)^3$ and $X_1(p)^3$ by permuting the coordinates. This induces a left action of $S_3$ on the various cycles given by 
\[
\sigma\cdot \Delta_{(x_1, x_2, x_3)}:=\sigma\circ \varphi_{x_1, x_2, x_3}(X(p))=\Delta_{(x_{\sigma(1)}, x_{\sigma(2)}, x_{\sigma(3)})}.
\]
(And similarly for cycles in $\widetilde{\mathcal{C}}$.)
Let $[x_1, x_2, x_3]\in \widetilde{I}^\times$ with determinant $(a, b, c)\in (\F_p^\times)^3$. For all $\sigma\in S_3$, observe that 
\[
\prod_{i=1}^3 \Det([x_{\sigma(1)}, x_{\sigma(2)}, x_{\sigma(3)}])_i=\sgn(\sigma) abc,
\]
where $\sgn(\sigma)$ is the sign of the permutation $\sigma$. The following result now follows from Lemma \ref{coro3}.
\begin{proposition}\label{lemma:S3}
If $p \equiv 1\pmod 4$, then the action of $S_3$ fixes $\Delta_+$ and $\Delta_-$. If $p \equiv 3\pmod 4$, then any transposition in $S_3$ permutes $\Delta_+$ and $\Delta_-$.
\end{proposition}

\subsection{Homological triviality}

We give two proofs, the first of which uses the Gross--Schoen projector \cite{grsc}.

\begin{definition}\label{def:pgks}
Let $C$ be a smooth projective geometrically connected curve over a number field $k$ and let $e$ be a $k$-rational point of $Y$. For any non-empty subset $T$ of $\{ 1, 2, 3 \}$, let $T'$ denote the complementary set. Write $p_T : C^3\lra C^{\vert T \vert}$ for the natural projection map and let $q_T(e) : C^{\vert T\vert}\lra C^3$ denote the inclusion obtained by filling in the missing coordinates using the point $e$. Let $P_{T}(e)$ denote the graph of the morphism $q_{T}(e)\circ p_T : C^3 \lra C^3$ viewed as a codimension $3$ cycle on the product $C^3\times C^3$. Define the Gross--Schoen projector by 
\[
P_{\gks}(e):=\sum_T (-1)^{\vert T'\vert} P_T(e) \in \CH^3(C^3\times C^3),
\]
where the sum is taken over all subsets of $\{ 1, 2, 3 \}$. This is an idempotent in the ring of correspondences of $C^3$ by \cite[Proposition 2.3]{grsc} with the property that it annihilates the cohomology groups $H^{i}(C^3(\C), \Z)$ for $i\in \{ 4,5,6 \}$ and maps $H^3(C^3(\C), \Z)$ onto the K\"unneth summand $H^1(C(\C), \Z)^{\otimes 3}$ by \cite[Corollary 2.6]{grsc}.
\end{definition}

Consider the Gross--Schoen projector on $X_0(p)^3$, with base-point some rational point $e\in X_0(p)(\Q)$, e.g., take $e$ to be one of the two cusps $\infty$ and $0$ of $X_0(p)$.\footnote{Mazur \cite[Theorem 1]{mazur} proved that $X_0(p)(\Q) = \{\infty, 0 \}$ whenever $p \not\in \{37, 43, 67, 163\}$ and the genus of $X_0(p)$ is greater than $1$. Moreover, the modular curve $X_0(37)$ has two non-cuspidal $\Q$-rational points, while $X_0(p)$ has a
unique non-cuspidal $\Q$-rational point when $p \in \{43, 67, 163\}$.} This idempotent correspondence acts on cohomology and annihilates $H^4(X_0(p)^3(\C), \Z)$, the target of the Betti cycle class map $\cl^2_B \colon \CH^2(X_0(p)^3_{\bar{\Q}}) \lra H^4(X_0(p)^3(\C), \Z)$. Hence, for any cycle $Z\in \CH^2(X_0(p)^3_{\bar{\Q}})$, the cycle $P_{\gks}(e)_*(Z)$ is null-homologous and belongs to $\CH^2(X_0(p)^3_{\bar{\Q}})_0$.

\begin{theorem}\label{prop:xihom}
The Darmon--Rotger cycle $\Xi=\Delta_+ - \Delta_-$ satisfies the equality $\Xi=P_{\gks}(e)_*(\Xi)$ for any base-point $e\in X_0(\Q)$. In particular, it is null-homologous. 
\end{theorem}

We will use the following lemma.

\begin{lemma}\label{lem:prphi}
Let $i<j \in \{ 1,2,3 \}$ and denote by $\pr_{ij}: X_0(p)^3 \lra X_0(p)^2$ the natural projection to the product of the $i$-th
 and $j$-th components. There exist elements $[x_1, x_2, x_3]$ and $[y_1, y_2, y_3]$ of $\widetilde{I}^\times$ satisfying
\[
\prod_{k=1}^3 \Det([x_1, x_2, x_3])_k \in (\F_p^\times)^{(2)} \qquad \text{ and } \qquad \prod_{k=1}^3 \Det([y_1, y_2, y_3])_k \not\in (\F_p^\times)^{(2)},
\]
and such that there is an equality
$\pr_{ij}\circ \varphi_{(x_1, x_2, x_3)}=\pr_{ij}\circ \varphi_{(y_1, y_2, y_3)}$
of maps $X(p)\lra X_0(p)^2$.
 \end{lemma}
 
 \begin{proof}
Fix some $a\not\in (\F_p^\times)^{(2)}$.\\
If $i=1$ and $j=2$, then we may take 
\[
(x_1, x_2, x_3)=((1,0), (0,1), (-1, -1)) \qquad \text{ and } \qquad (y_1, y_2, y_3)=((1,0), (0,1), (-a, -1)).
\] 
If $i=1$ and $j=3$, then we may take 
\[
(x_1, x_2, x_3)=((-1,0), (1,-1), (0, 1)) \qquad \text{ and } \qquad (y_1, y_2, y_3)=((-1,0), (a,-1), (0, 1)).
\] 
If $i=2$ and $j=3$, then we may take 
\[
(x_1, x_2, x_3)=((-1,-1), (1,0), (0, 1)) \qquad \text{ and } \qquad (y_1, y_2, y_3)=((-1,-a), (1,0), (0, 1)).
\] 
 \end{proof}
 
 \begin{remark}\label{rem:phipm}
 The maps $\varphi_{(x_1, x_2, x_3)}$ and $\varphi_{(y_1, y_2, y_3)}$ associated with the specific choices made in the above proof will be denoted $\varphi_{+}(ij)$ and $\varphi_{-}(ij)=\varphi_{-}(ij; a)$ respectively.
 \end{remark}

\begin{proof}[Proof of Theorem \ref{prop:xihom}]
Observe that 
\[
P_{\gks}(e)_*(\Xi)=\Xi - P_{12}(e)_*(\Xi) -P_{13}(e)_*(\Xi) - P_{23}(e)_*(\Xi)  +P_{1}(e)_*(\Xi)  + P_{2}(e)_*(\Xi) +P_{3}(e)_*(\Xi). 
\]

Let $i<j \in \{ 1,2,3\}$ and consider $P_{ij}(e)_*(\Xi)$. Let $k\in \{ 1,2,3\}$ be the remaining element distinct from $i$ and $j$. The correspondence $P_{ij}(e)$ is the graph of the function
$$q_{ij}(e) \circ \pr_{ij} : X_0(p)^3 \lra X_0(p)^3,$$ which replaces the $k$-th coordinate by the point $e$, 
and $P_{ij}(e)_*(\Xi)$ is the image of $\Xi$ under $q_{ij}(e) \circ \pr_{ij}$.  
Choose $[x_1, x_2, x_3]$ and $[y_1, y_2, y_3]$ of $\widetilde{I}^\times$ satisfying the properties of Lemma \ref{lem:prphi} for the fixed $i$ and $j$.
The first condition ensures that 
$$\varphi_{(x_1, x_2, x_3)}(X(p))=\Delta_+ \quad \text{ and } \quad \varphi_{(y_1, y_2, y_3)}(X(p))=\Delta_-,$$ while the second condition implies that
\[
P_{ij}(e)_*(\Delta_+)=q_{ij}(e) \circ \pr_{ij}\circ \varphi_{(x_1, x_2, x_3)}(X(p))=q_{ij}(e) \circ \pr_{ij}\circ \varphi_{(y_1, y_2, y_3)}(X(p))=P_{ij}(e)_*(\Delta_-).
\]
As a consequence, we have $P_{ij}(e)_*(\Xi)=0$.

Let $i\in \{ 1,2,3 \}$ and consider $P_{i}(e)_*(\Xi)$. Let $j, k\in \{ 1,2,3 \}$ such that $\{ i, j, k \}=\{ 1,2,3 \}$. The correspondence $P_{i}(e)$ is the graph of the map
$q_{i}(e) \circ \pr_{i} : X_0(p)^3 \lra X_0(p)^3$, which replaces the $j$-th and $k$-th coordinates by the point $e$,
and $P_{i}(e)_*(\Xi)$ is the image of $\Xi$ under $q_{i}(e) \circ \pr_{i}$. This map can be written as the composition
\[
q_{i}(e) \circ \pr_{i}=(q_{ik}(e) \circ \pr_{ik})\circ (q_{ij}(e) \circ \pr_{ij}),
\]
hence in terms of correspondences we have $P_{i}(e)=P_{ik}(e)\circ P_{ij}(e)$.
It follows from the previous paragraph that 
$P_{i}(e)_*(\Xi)=0$.

We conclude that $\Xi=P_{\gks}(e)_*(\Xi)$. 
\end{proof}

\begin{remark}
A perhaps more direct way to see that the cycle $\Xi$ is null-homologous is to consider its image under the de Rham cycle class map, namely $$\cl^2_{\dR}(\Xi)=\cl^2_{\dR}(\Delta_+)-\cl^2_{\dR}(\Delta_-)\in H^4_{\dR}(X_0(p)^3/\C),$$ where we recall that
\[
\int_{X_0(p)(\C)^3} \cl^2_{\dR}(\Delta_{\pm})\wedge \alpha = \int_{\Delta_{\pm}} \alpha, \qquad \text{ for all } \alpha \in H^2_{\dR}(X_0(p)^3/\C).
\]
By the K\"unneth decomposition for $H^2_{\dR}(X_0(p)^3/\C)$, any component of $\alpha$ can at most involve de Rham classes coming from two of the three components of $X_0(p)^3$; indeed, the components are either of the form $\pr_i^*(\beta)$ for some $\beta\in H^2_{\dR}(X_0(p)/\C)$ and $i\in \{ 1,2,3\}$, or of the form $\pr_j^*(\gamma)\wedge \pr_k^*(\delta)$ for some $\gamma, \delta\in H^1_{\dR}(X_0(p)/\C)$ and $j< k\in \{ 1,2,3\}$. Using the notations of Remark \ref{rem:phipm}, observe that 
\begin{align*}
\int_{\Delta_{\pm}} \pr_i^*(\beta) & = \int_{X(p)} (\pr_{i}\circ\varphi_{\pm}(ij))^*(\beta) \\
\int_{\Delta_{\pm}} \pr_j^*(\gamma)\wedge \pr_k^*(\delta) & =\int_{X(p)} (\pr_{jk}\circ\varphi_{\pm}(jk))^*(\gamma\wedge \delta).
\end{align*}
Since 
\begin{align*}
\pr_i \circ \varphi_+(ij) & = \pr_i\circ \varphi_-(ij): X(p)\lra X_0(p) \\
\pr_{jk} \circ \varphi_+(jk) & = \pr_{jk}\circ \varphi_-(jk) : X(p)\lra X_0(p)^2,
\end{align*}
this implies that $\cl^2_{\dR}(\Delta_+)=\cl^2_{\dR}(\Delta_-)$ in $H^4_{\dR}(X_0(p)^3/\C)$.
\end{remark}

\section{Weil--Deligne representations and local factors}\label{s:WD}

This section provides background material on Weil--Deligne representations and epsilon factors following \cite{del73, rohrlich94}.

\subsection{Weil--Deligne representations}\label{s:WDG}

Let $q$ denote a prime number. The embedding $\bar{\Q}\hookrightarrow \bar{\Q}_q$ fixed in Section \ref{not:emb} realises $\Gal(\bar{\Q}_q/\Q_q)$ as the decomposition subgroup at $q$ of $\Gal(\bar{\Q}/\Q)$. It sits in the short exact sequence
\[
1\lra I_q\lra \Gal(\bar{\Q}_q/\Q_q)\overset{r}{\lra} \Gal(\bar{\F}_q/\F_q)\lra 1
\]
where $I_q$ denotes the inertia subgroup at $q$ and $r$ denotes the natural reduction map. The group $\Gal(\bar{\F}_q/\F_q)$ is topologically generated by the Frobenius automorphism $\phi : x\mapsto x^q$ and is isomorphic to the profinite completion $\hat{\Z}$ of $\Z$. We denote by $\varphi$ the inverse of the Frobenius automorphism $\phi$. The Weil group at $q$, denoted $W(\bar{\Q}_q/\Q_q)$, is defined as the pre-image under $r$ of the infinite cyclic subgroup of $\Gal(\bar{\F}_q/\F_q)$ generated by $\phi$. We endow it with the coarsest topology for which $r : W(\bar{\Q}_q/\Q_q)\lra \langle \phi\rangle$ and $I_q \hookrightarrow W(\bar{\Q}_q/\Q_q)$ are both continuous and for which $W(\bar{\Q}_q/\Q_q)$ is a topological group.
A representation of the Weil group is a continuous homomorphism of groups 
\[ \sigma_q : W(\bar{\Q}_q/\Q_q)\lra \GL(V) \] 
where $V$ is a finite-dimensional complex vector space.

\begin{example}
Examples of Weil representations include all finite-dimensional complex representations of Galois groups of finite extensions of $\Q$. Also, we identify all characters of $\Q_q^\times$ with characters of $W(\bar{\Q}_q/\Q_q)$ via the Artin isomorphism 
\begin{equation}\label{eq:artin}
\Q_q^\times \simeq W(\bar{\Q}_q/\Q_q)^{\ab}
\end{equation}
normalised so that it maps $q$ to the image in $W(\bar{\Q}_q/\Q_q)^{\ab}$ of an inverse Frobenius element of $W(\bar{\Q}_q/\Q_q)$. Another example of a Weil representation is given by the character 
\begin{equation}\label{def:omega}
    \omega_q : W(\bar{\Q}_q/\Q_q)\lra \C^\times
\end{equation} 
defined by $\omega_q(I_q)=1$ (i.e., it is unramified) and $\omega_q(\Phi)=q^{-1}$ where $\Phi$ is an inverse Frobenius element of $\Gal(\bar{\Q}_q/\Q_q)$ (i.e., an element satisfying $r(\Phi)=\varphi$). Under the isomorphism (\ref{eq:artin}) the character $\omega_q$ corresponds to the $q$-adic norm character $\Vert \cdot \Vert_q : \Q_q^\times \lra \C^\times$ normalised such that $\Vert q\Vert_q=q^{-1}$.
\end{example}

\begin{definition}\label{def:WDrep}
A Weil--Deligne representation is a pair $\sigma'_q=(\sigma_q, N_q)$ where $\sigma_q$ is a Weil representation on a finite-dimensional complex vector space $V$ and $N_q$ is a nilpotent endomorphism of $V$ satisfying 
\begin{equation}
\sigma_q(g)\circ N_q\circ \sigma_q(g)^{-1}=\omega_q(g)N_q, \qquad \text{ for all } g\in W(\bar{\Q}_q/\Q_q).
\end{equation}
\end{definition}

There is also a theory of Weil--Deligne representations at archimedean places, but we will not need it in our calculations. 

\begin{example}\label{ex:cyc}
Fix an embedding $\iota: \Q_\ell\hookrightarrow \C$.
Consider the $\ell$-adic cyclotomic character 
\[
\omega_{\cyc, \ell} : \Gal(\bar{\Q}/\Q)\twoheadrightarrow \Gal(\Q(\zeta_{\ell^{\infty}})/\Q)\lra \Z_\ell^\times
\]
where $\zeta_{\ell^{\infty}}$ denotes a compatible system $(\zeta_{\ell^n})_n$ of primitive $\ell^n$-th roots of unity. If $\sigma$ is an element in $\Gal(\bar{\Q}/\Q)$, then $\sigma(\zeta_{\ell^n})=\zeta_{\ell^n}^{m_n}$ for some compatible $m_n\in (\Z/\ell^n\Z)^\times$ and $\omega_{\cyc, \ell}(\sigma)=(m_n)_n\in \Z_\ell^\times$. This character is unramified at $q$ since the extension $\Q(\zeta_{\ell^{\infty}})$ of $\Q$ is ramified only at $\ell$. The Weil--Deligne representation at $q$ of $\omega_{\cyc, \ell}$ is then the Weil representation $\iota\circ\omega_{\cyc, \ell}\vert_{W(\bar{\Q}_q/\Q_q)}$. If $\Phi$ is a geometric Frobenius element of $W(\bar{\Q}_q/\Q_q)$, then $\omega_{\cyc, \ell}(\Phi)=q
^{-1}\in \Z_\ell^{\times}$ and thus $$\iota\circ\omega_{\cyc, \ell}\vert_{W(\bar{\Q}_q/\Q_q)}=\omega_q,$$ where $\omega_q$ is defined by \eqref{def:omega}. In particular, the Weil--Deligne representation of $\omega_{\cyc, \ell}$ at $q$ is independent of $\iota$ and $\ell$.
\end{example}

\begin{example}\label{def:sp2}
Let $( e_0, e_1 )$ denote the standard basis of $\C^2$. The special representation of the Weil--Deligne group at $q$ of dimension $2$, denoted $\spe(2)$, is the representation $(\sigma_q, N)$ defined by the matrices 
\[
    \sigma_q:=\left(\begin{matrix} 1 & 0 \\ 0 & \omega_q \end{matrix}\right) \qquad \text{ and } \qquad N:=\left(\begin{matrix} 0 & 0 \\ 1 & 0 \end{matrix}\right).
\]
\end{example}

\begin{example}\label{ex:3}
Let $f = \sum_{n\geq 1} a_n(f)q^n\in S_2(\Gamma_0(p))$ be a normalised cuspidal eigenform.
Let $\mathfrak{l}$ denote the prime of $K_{f}$ above $\ell$ determined by the field embeddings fixed in Section \ref{not:emb}. Attached to this data is a $2$-dimensional $\mathfrak{l}$-adic Galois representation 
\begin{equation}\label{rep:f}
V_{\ell}(f) : \Gal(\bar{\Q}/\Q) \lra \GL_2(K_{f, \mathfrak{l}}).
\end{equation}
See \cite[Theorem 3.1]{ddt} for the precise definition and properties. Note that we suppressed the dependencies of the various embeddings from the notation, as these have been fixed from the beginning. Let $q$ be a prime different from $\ell$ and fix an embedding $\iota_\ell : K_{f, \mathfrak{l}} \hookrightarrow \C$. Following \cite[\S 4]{rohrlich94}, one can associate to $V_\ell(f)$ a $2$-dimensional Weil--Deligne representation $\sigma'_{f, \ell, \iota_\ell, q}=(\sigma_{f, \ell, \iota_\ell, q}, N_{f, \ell, \iota_\ell, q})$. It turns out that the isomorphism class of the representation $\sigma'_{f, \ell, \iota_\ell, q}$ is independent of $\ell$ and $\iota_\ell$ and we shall simply write $\sigma'_{f, q}=(\sigma_{f, q}, N_{f, q})$. This is the Weil--Deligne representation of $f$ at $q$. 
\begin{proposition}\label{WDf}
The Weil--Deligne representations of $f$ satisfy the following: 
\begin{itemize}
\item If $q\neq p$, then $N_{f, q}=0$ and 
$
\sigma_{f, q}\simeq\xi_q\oplus \xi_q^{-1}\omega_q^{-1}
$
for some unramified character $\xi_q$. Here $\omega_q$ is the Weil--Deligne representation of the $\ell$-adic cyclotomic character defined by \eqref{def:omega} and Example \ref{ex:cyc}.
\item Let $\lambda$ be the unramified quadratic character of $W(\bar{\Q}_p/\Q_p)$ defined by $\lambda(\Phi)=a_p(f)$, where $\Phi$ denotes an inverse Frobenius element. Then 
$
\sigma_{f, p}' \simeq \lambda \omega_q^{-1}\otimes \spe(2),
$
so that, in particular, $N_{f, q}\neq 0$ and $\sigma_{f, q}'$ is ramified. Here $\spe(2)$ is the special representation in Example \ref{def:sp2}.
\end{itemize}
\end{proposition}

\begin{proof}
Using \cite[Theorem 3.1]{ddt}, the proofs in \cite[\S 14, \S15]{rohrlich94} adapt to this setting and give the above descriptions of the Weil--Deligne representations of $f$. In particular, these are independent of the choices of a prime $\ell$ and an embedding $\iota_\ell : K_{f, \mathfrak{l}} \hookrightarrow \C$.
\end{proof}
\end{example}

\subsection{Local factors}\label{s:epsilon}

Epsilon factors were first introduced by Deligne and their properties are summarised in \cite[\S 5]{del73}. We will follow the exposition in \cite{rohrlich94} to collect the essential properties needed for the purposes of this paper. 

If $q$ is a finite place, let $\sigma_q'=(\sigma_q, N_q)$ be a Weil--Deligne representation with associated finite-dimensional complex vector space $V$. Let $\psi_q : \Q_q\lra \C^\times$ denote an additive character and let $\d x_q$ denote the choice of a Haar measure on $\Q_p$. The epsilon factor associated to $\sigma_q'$ depends on $\psi_q$ and $\d x_q$ and is given by 
\begin{equation}\label{eqprime}
\epsilon'(\sigma_q', \psi_q, \d x_q):=\epsilon(\sigma_q, \psi_q, \d x_q)\delta(\sigma_q')\in \C^\times,
\end{equation}
where
\begin{equation}\label{def:deltawd} 
\delta(\sigma_q'):=\det(-\Phi \: \vert \: V^{I_q}/(V^{I_q}\cap \ker N_q)),
\end{equation}
and $\epsilon(\sigma_q, \psi_q, \d x_q)$ is the epsilon factor of the Weil representation $\sigma_q$, which we will now describe.

Explicit formulas for the epsilon factor of a character at a finite place are given as follows. Let $\mu$ be a character of $\Q_q^\times$ identified with a $1$-dimensional representation of the Weil group via \eqref{eq:artin}. Let $n(\psi_q)$ denote the largest integer $n$ such that $\psi_q$ is trivial on $q^{-n}\Z_q$. Let $a(\mu)$ denote the conductor of $\mu$, i.e., $a(\mu)=0$ if $\mu$ is unramified and otherwise $a(\mu)$ is the smallest positive integer $m$ such that $\mu$ is trivial on $1+q^m\Z_q$. Then 
    \begin{equation}\label{epsQq}
    \epsilon(\mu, \psi_q, \d x_q)=
    \begin{cases}
    \int_{q^{-(n(\psi_q)+a(\mu))}\Z_q^\times} \mu^{-1}(x) \psi_q(x) \d x_q & \text{ if $\mu$ is ramified} \\
    \mu \omega_q^{-1}(q^{n(\psi_q)}) \int_{\Z_q} \d x_q & \text{ if $\mu$ is unramified}.
\end{cases}
\end{equation} 

The epsilon factor of a Weil representation is completely determined by the following result. 
\begin{theorem}\label{prop:eps}
Let $k$ be either $\R, \C$ or $\Q_q$ for some finite place $q$. There is a unique function $\epsilon$, which to any Weil representation $\sigma$, any non-trivial additive character $\psi : k\lra \C^\times$ and any choice of a Haar measure $\d x$ on $k$, associates a complex number $\epsilon(\sigma, \psi, \d x)\in \C^\times$ satisfying:
\begin{itemize}
\item[$i)$] $\epsilon(*, \psi, \d x)$ is multiplicative in short exact sequences.
\item[$ii)$] If $L/k$ is any finite extension of $k$ in $\bar{k}$ and $\sigma_L$ is a Weil representation of $L$, then for any choice of Haar measure $\d x_L$ on $L$, we have 
\[
\epsilon\bigg(\ind_{W(\bar{k}/L)}^{W(\bar{k}/k)} \sigma_L, \psi, \d x\bigg)=\epsilon\bigg(\sigma_L, \psi\circ\tr_{L/k}, \d x_L\bigg)\left(\frac{\epsilon(\ind_{W(\bar{k}/L)}^{W(\bar{k}/k)} 1_L, \psi, \d x)}{\epsilon(1_L, \psi\circ\tr_{L/k}, \d x_L)}\right)^{\dim \sigma_L}.
\]
\item[$iii)$] If $\dim \sigma=1$, then $\epsilon(\sigma, \psi, \d x)$ is given by explicit formulas (formula (\ref{epsQq}) in the case $k=\Q_q$, and see \cite{del73} for the formulas in the archimedean case, which we will not need).
\end{itemize}
\end{theorem} 

\begin{proof}
This is \cite[Theorem 4.1]{del73}.
\end{proof}

\begin{definition}
Let $k$ be either $\R, \C$ or $\Q_q$ for some finite place $q$. Given a Weil--Deligne representation $\sigma'=(\sigma, N)$ of $k$, the choice of an additive character $\psi : k\lra \C^\times$ and a Haar measure $\d x$ on $k$, we define the root number 
\[
W(\sigma', \psi):=\frac{\epsilon'(\sigma', \psi, \d x)}{\vert \epsilon'(\sigma', \psi, \d x)\vert}.
\]
\end{definition}

As the notation suggests, the root number is independent of the choice of a Haar measure $\d x$, as can be seen from \cite[\S 11 Proposition (ii)]{rohrlich94}. Moreover, if the Weil--Deligne representation $\sigma'_q$ at a finite prime $q$ is essentially symplectic, then the local root number at $q$ is independent of the additive character $\psi$ and belongs to $\{ \pm 1 \}$ by \cite[\S 12]{rohrlich94}. We shall simply write $W(\sigma'_q)$ in this case.

We end this subsection by collecting a few results concerning epsilon factors of Weil representations at finite places.

\begin{proposition}\label{propeps:cyctwist}
Let $\sigma_q$ be a Weil representation at a prime $q$.
If $\mu$ is an unramified character of $\Q_q^\times$, $\psi : \Q_q\lra \C^\times$ is a non-trivial additive character, and $\d x$ is Haar measure on $\Q_q$, then 
\[ 
\epsilon(\sigma_q\otimes \mu, \psi, \d x)=\mu(q^{n(\psi)\dim (\sigma_q)+a(\sigma_q)})\epsilon(\sigma_q, \psi, \d x).
\]
Here $a(\sigma_q)$ is the conductor of $\sigma_q$ defined in \cite[\S 10]{rohrlich94}.
\end{proposition}

\begin{proof}
This is \cite[\S 11 Proposition (iii)]{rohrlich94}.
\end{proof}

The following proposition gives an explicit formula for the epsilon factor of a ramified character of conductor $1$. Note that if $\psi : \Q_q\lra \C^\times$ is an additive character with $n(\psi)=0$, then $\psi\vert_{\Z_q^\times}=1$ but $\psi\vert_{q^{-1}\Z_q^\times}\neq 1$. Thus there exists $c\in \F_q^\times$ such that $\psi(1/q)=\exp((2\pi i c)/q)$. In this case, we write $\psi_c$ for $\psi$. 

\begin{proposition}\label{form:eps}
Let $\mu$ be a ramified character of $\Q_q^\times$ identified with a $1$-dimensional representation of the Weil group via \eqref{eq:artin}. Let $\psi : \Q_q\lra \C^\times$ denote an unramified additive character, i.e., $n(\psi)=0$, and $\d x$ denote the Haar measure on $\Q_p$ such that $\int_{\Z_p} \d x=1$. Suppose that $a(\mu)=1$. Let $c\in \F_q^\times$ such that $\psi=\psi_c$. Then the following formula holds: 
\[
\epsilon(\mu, \psi, \d x)=\mu(c)\mu(q) G(\mu^{-1})
\]
where $G(\mu^{-1})=\sum_{b\in \F_q^\times} \mu^{-1}(b)e^{\frac{2\pi i b}{q}}$ is the Gauss sum of the character $\mu^{-1}$.
\end{proposition}

\begin{proof}
This can be extracted from the proof of \cite[Theorem 3.2 (2)]{pacetti}.
\end{proof}

\begin{corollary}\label{coro:minusone}
With the same notations and assumptions as in Proposition \ref{form:eps}, we have the formula 
\[
\epsilon(\mu, \psi, \d x)\epsilon(\mu^{-1}, \psi, \d x)=q\mu(-1).
\]
\end{corollary}

\begin{proof}
Applying the result of Proposition \ref{form:eps} to $\mu$ and $\mu^{-1}$ leads to
\[
\epsilon(\mu, \psi, \d x)\epsilon(\mu^{-1}, \psi, \d x)=G(\mu^{-1})G(\mu).
\]
By standard properties of Gauss sums, we have $G(\mu^{-1})=\mu(-1)\overline{G(\mu)}$. Using the fact that $\vert G(\mu)\vert^2=q$, we obtain the desired result.
\end{proof}

\section{Triple product root numbers}\label{s:trip}

We work with the modular curve $X_0(p)$ of prime level and assume that the genus of $X_0(p)$ is positive (i.e., $p=11$ or $p>13$). 
Let 
\[
f_1=\sum_{n\geq 1} a_n(f_1)q^n, \qquad f_2=\sum_{n\geq 1} a_n(f_2)q^n, \qquad f_3=\sum_{n\geq 1} a_n(f_3)q^n
\]
be three normalised cuspidal eigenforms of level $\Gamma_0(p)$, and let $F:=f_1\otimes f_2\otimes f_3$.  
Recall that $\chi$ denotes the Legendre symbol at $p$. Let $L(F,\chi,s)$ be the $L$-function associated to the compatible family of $8$-dimensional $\ell$-adic Galois representations 
\[
\{ V_\ell(F,\chi)=V_\ell(f_1)\otimes V_\ell(f_2)\otimes V_\ell(f_3) \otimes \chi \}_\ell,
\]
where the representations $V_\ell(f_{i})$ for $i\in \{ 1,2,3 \}$ are the ones in \eqref{rep:f}.  

Following the recipe in \cite{TateThesis}, one may lift $\chi$ to a unitary Hecke character $\chi_{\bA} :  \bA_\Q^\times/\Q^\times \lra \C^\times$ by setting $\chi_{\bA}(g)=\prod_{v} \chi_v(g_v)$ where $v$ runs over all places of $\Q$ and 
\[
    \chi_\infty(g_\infty)=
    \begin{cases}
    1 & \text{ if } \chi(-1)=1 \\
    1 & \text{ if } \chi(-1)=-1, g_\infty>0 \\
    -1 & \text{ if } \chi(-1)=-1, g_\infty<0
    \end{cases} \qquad  \chi_\ell(g_\ell)=
    \begin{cases}
    \chi(\ell)^{\ord_\ell(g_\ell)} & \text{ if } \ell\neq p \\
    \chi(j)^{-1} & \text{ if } g_p\in p^k(j+p\Z_p). 
    \end{cases}
\]
The collection of $\ell$-adic characters $\{ \chi_\ell : \Q_\ell
^\times \lra \C^\times\}_\ell$ is characterised by the following:
\begin{itemize}
    \item For $\ell\neq p$, $\chi_\ell$ is unramified with $\chi_\ell(\ell)=\left( \frac{\ell}{p} \right)$;
    \item $\chi_p$ is tamely ramified, $\chi_p(p)=1$, and $\chi_p\vert_{\Z_p^\times}=\left( \frac{\cdot}{p} \right)$.
\end{itemize}
The Weil--Deligne representation of $F\otimes \chi$ at a prime $q$ is the $8$-dimensional representation 
\[
\sigma'_{F,\chi,q}=\sigma'_{f_1, q}\otimes \sigma'_{f_2, q} \otimes \sigma'_{f_3, q}\otimes \chi_q,
\]
where $\sigma'_{f_i, q}$ is described in Example \ref{ex:3}.
Concretely, we have
\[
\sigma'_{F,\chi,q}=(\sigma_{F,\chi, q}, N_{F,\chi, q})=(\sigma_{f_1, q}\otimes \sigma_{f_2, q} \otimes \sigma_{f_3, q} \otimes \chi_q, N_{f_1, q}\otimes 1\otimes 1+1\otimes N_{f_2, q}\otimes 1 +1\otimes 1\otimes N_{f_3, q}).
\]
Following the recipe in \cite{del73}, we attach to $F\otimes \chi$ a completed $L$-function 
\[
\Lambda(F,\chi,s):=2^4(2\pi)^{3-4s}\Gamma(s-1)^3\Gamma(s)L(F,\chi,s).
\]
(This is based on the fact that the Hodge numbers of $F\otimes \chi$ are given by 
$h^{3,0}=1$ and $h^{2,1}=3$.)

Following general recipes for motives \cite{del73}, we define the conductor, the global epsilon factor, and the global root number of $F\otimes \chi$.
The conductor of $F\otimes \chi$ is defined to be
\begin{equation}\label{def:cond}
\cond(F,\chi):=\prod_q q^{a(\sigma'_{F,\chi, q})}\in \mathbb{N},
\end{equation}
where the product is over all finite places $q$. 
Consider $\psi=\prod_v \psi_v : \A_\Q/\Q \lra \C$ an additive character of the ad\`eles and let $\d x$ denote the normalised Haar measure on the ad\`eles such that $\int_{\A_\Q/\Q} \d x=1$. It decomposes as a product of local Haar measures $\d x_v$  which satisfy $\int_{\Z_v} \d x_v=1$ for almost all finite places $v$. The global epsilon factor of $F\otimes \chi$ is then defined to be 
\[
\epsilon(F,\chi):=\prod_v \epsilon'(\sigma'_{F,\chi, v}, \psi_v, \d x_v),
\]
which is independent of the choice of $\psi$ and $\d x$. Moreover, $\epsilon'(\sigma'_{F,\chi, v}, \psi_v, \d x_v)=1$ for almost all $v$ (in fact, for all $v\neq \infty, p$ as we will see in the proof of Theorem \ref{roottripleEram} below).
The global root number is defined to be 
\begin{equation}\label{def:grn}
    W(F,\chi)=\prod_v W(\sigma'_{F,\chi, v}, \psi_v, \d x_v).
\end{equation} 
The completed $L$-function 
\begin{equation}\label{compL}
\Lambda^*(F,\chi, s):=\cond(F,\chi)^{\frac{s}{2}}\Lambda(F,\chi, s)
\end{equation}
admits analytic continuation to the entire complex plane and satisfies the functional equation 
\begin{equation}\label{eq:FE}
\Lambda^*(F, \chi, s)=W(F,\chi)\Lambda^*(F, \chi, 4-s).
\end{equation}

\begin{remark}\label{rem:psq}
    Notice that $F\otimes \chi$ is equal to the tensor product of the three newforms $f_1, f_2$ and $f_3^{(p)}$, where $f_3^{(p)}=f_3 \otimes \chi$. The $L$-function $\Lambda^*(F,\chi,s)$ is the triple product $L$-function associated to the triple $(f_1, f_2, f_3^{(p)})$. The first two forms have level $\Gamma_0(p)$ while the form $f_3^{(p)}$ has level $\Gamma_0(p^2)$. Hence, the analytic properties and functional equation of $\Lambda^*(F, \chi, s)$ fall outside the scope of \cite{grosskudla} where the case of three newforms of the same square-free level is treated. However, as explained in \cite{harriskudla}, the analytic properties and functional equation in this case follow from \cite{PSR}.
\end{remark}

\begin{theorem}\label{roottripleEram}
$W(F,\chi)=-1.$
\end{theorem}

\begin{proof}
The local root number at infinity of $F\otimes \chi$ is the same as the one of $F$ since twisting by finite order characters does not affect Hodge structures. In particular, it is equal to $-1$ by \cite{grosskudla}. We thus focus on the local root numbers at the finite places. For any prime $\ell$, we choose an additive character $\psi_\ell$ with $n(\psi_\ell)=0$, as well as the Haar measure $\d x_\ell$ normalised such that $\int_{\Z_\ell} \d x_\ell=1$.

Let $q$ be a prime distinct from $p$. By Proposition \ref{WDf} we have, for $i\in \{ 1,2,3 \}$, 
\[
\sigma'_{f_i, q}=\sigma_{f_i, q}=\xi_{i, q}\oplus \xi_{i, q}^{-1}\omega_q^{-1}
\] 
for some unramified characters $\xi_{i, q}$. The character $\chi_q$ is also unramified. We therefore obtain  
\begin{multline*}
    \sigma'_{F,\chi, q}=\sigma_{F,\chi, q}=\xi_{1,q}\xi_{2,q}\xi_{3,q}\chi_q
\oplus \xi_{1,q}\xi_{2,q}^{-1}\xi_{3,q}\omega_q^{-1}\chi_q
\oplus \xi_{1,q}^{-1}\xi_{2,q}\xi_{3,q}\omega_q^{-1}\chi_q
\oplus \xi_{1,q}^{-1}\xi_{2,q}^{-1}\xi_{3,q}\omega_q^{-2}\chi_q \\
\oplus \xi_{1,q}\xi_{2,q}\xi_{3,q}^{-1}\omega_q^{-1}\chi_q 
\oplus \xi_{1,q}\xi_{2,q}^{-1}\xi_{3,q}^{-1}\omega_q^{-2}\chi_q
\oplus \xi_{1,q}^{-1}\xi_{2,q}\xi_{3,q}^{-1}\omega_q^{-2}\chi_q
\oplus \xi_{1,q}^{-1}\xi_{2,q}^{-1}\xi_{3,q}^{-1}\omega_q^{-3}\chi_q.
\end{multline*}
Since all characters involved are unramified, Theorem \ref{prop:eps} $i)$ and \eqref{epsQq} imply, given the choice of $\psi_q$ and $\d x_q$, that
$\epsilon'(\sigma'_{F,\chi,q}, \psi_q, \d x_q)=1,$
and in particular $W(\sigma'_{F,\chi,q})=1$.

For each $i\in \{ 1,2,3 \}$, let $\lambda_i$ be the unramified quadratic character of $W(\bar{\Q}_p/\Q_p)$ defined by $\lambda_i(\Phi)=a_p(f_i)$, where $\Phi$ denotes an inverse Frobenius element. We will sometimes view it as a character of $\Q_p^\times$ via the identification \eqref{eq:artin}. Let $\lambda=\lambda_1\lambda_2\lambda_3$ denote the product of these characters.
By Proposition \ref{WDf}, the Weil--Deligne representation of $F\otimes \chi$ at $p$ is given by 
$$\sigma'_{F,\chi,p}=\chi_p \lambda \omega_p^{-3}\otimes \spe(2)^{\otimes 3}.$$ 
Let $V$ denote the complex vector space associated to it.
The character $\chi_p$ is tamely ramified, i.e., $a(\chi_p)=1$. Suppose, by contradiction, that $V^{I_p}\neq 0$. Then there is a non-zero vector $v\in V$ which is fixed by the action of the inertia $I_p$. But
$\sigma_{F,\chi,p}(g)(v)=\chi_p(g)v$ for all $g\in I_p$
since $\sigma_{F,\chi,p}=\sigma_{f_1,p}\otimes\sigma_{f_2,p}\otimes\sigma_{f_3,p}\otimes \chi_p$ and $\sigma_{f_1,p}\otimes\sigma_{f_2,p}\otimes\sigma_{f_3,p}$ is unramified. As $v\in V^{I_p}$, we must have $\chi_p(g)v=v$ which implies that $\chi_p(g)=1$ since $v\neq 0$. Since this holds for all $g\in I_p$, it contradicts the fact that $\chi_p$ is ramified. Hence $V^{I_p}=0$ and as a consequence $\delta(\sigma'_{F,\chi, p})=1$, which implies that
\begin{equation}\label{epsss}
\epsilon'(\sigma'_{F\chi,p}, \psi_p, \d x_p)=\epsilon(\sigma_{F,\chi,p}, \psi_p, \d x_p).
\end{equation}
If $( e_0, e_1 )$ denotes the standard basis of $\C^2$, then $\spe(2)$ is the representation $(\sigma_p, N)$ defined in Example \ref{def:sp2} by the matrices 
\[
\sigma_p:=\left(\begin{matrix} 1 & 0 \\ 0 & \omega_p \end{matrix}\right) \qquad \text{ and } \qquad N:=\left(\begin{matrix} 0 & 0 \\ 1 & 0 \end{matrix}\right).
\]
Let us denote by $V_i=\C^2$ the complex vector space associated to $\sigma'_{f_i, p}$ and by $(e_0^{(i)}, e_1^{(i)})$ its standard basis for each $i\in \{ 1,2,3\}$. Then $V=V_1\otimes_{\C} V_2\otimes_{\C} V_3=\C^8$ is the space of $\sigma_{M, p}$ and an ordered basis for it is given by 
\begin{equation}\label{basis}
\begin{aligned}
B:=( e_0^{(1)}\otimes e_0^{(2)}\otimes e_0^{(3)}, e_0^{(1)}\otimes e_0^{(2)}\otimes e_1^{(3)}, e_0^{(1)}\otimes e_1^{(2)}\otimes e_0^{(3)}, e_0^{(1)}\otimes e_1^{(2)}\otimes e_1^{(3)}, \\
e_1^{(1)}\otimes e_0^{(2)}\otimes e_0^{(3)},  e_1^{(1)}\otimes e_0^{(2)}\otimes e_1^{(3)}, e_1^{(1)}\otimes e_1^{(2)}\otimes e_0^{(3)}, e_1^{(1)}\otimes e_1^{(2)}\otimes e_1^{(3)} ).
\end{aligned}
\end{equation}
With respect to the basis $B$, the representation 
\[
\spe(2)^{\otimes 3}=(\sigma_p^{\otimes 3}, N^{\otimes 3}:=N\otimes 1\otimes 1+1\otimes N\otimes 1+1\otimes 1 \otimes N)
\] 
is given by the matrices
\[
\sigma_p^{\otimes 3}=
\left(
\begin{matrix}
1       &       0       &       0       &       0       &       0       &       0       &       0       &       0       \\
0       & \omega_p   &       0       &       0       &       0       &       0       &       0       &       0       \\
0       &       0       &  \omega_p  &       0       &       0       &       0       &       0       &       0       \\
0       &       0       &       0      & \omega_p^2 &       0       &       0       &       0       &       0       \\
0       &       0       &       0       &       0       & \omega_p   &       0       &       0       &       0       \\
0       &       0       &       0       &       0       &       0       & \omega_p^2 &       0       &       0       \\
0       &       0       &       0       &       0       &       0       &       0       & \omega_p^2 &       0       \\
0       &       0       &       0       &       0       &       0       &       0       &       0       &       \omega_p^3    
\end{matrix}
\right)
\qquad \text{ and } \qquad
N^{\otimes 3}=
\left(
\begin{matrix}
0       &       0       &       0       &       0       &       0       &       0       &       0       &       0       \\
1       &       0       &       0       &       0       &       0       &       0       &       0       &       0       \\
1       &       0       &       0       &       0       &       0       &       0       &       0       &       0       \\
0       &       1       &       1       &       0       &       0       &       0       &       0       &       0       \\
1       &       0       &       0       &       0       &       0       &       0       &       0       &       0       \\
0       &       1       &       0       &       0       &       1       &       0       &       0       &       0       \\
0       &       0       &       1       &       0       &       1       &       0       &       0       &       0       \\
0       &       0       &       0       &       1       &       0       &       1       &       1       &       0       
\end{matrix}
\right).
\]
We conclude that 
\begin{equation}\label{weilE}
\sigma_{F,\chi, p}\simeq \chi_p\lambda\omega_p^{-3}\oplus \chi_p\lambda\omega_p^{-2}\oplus \chi_p\lambda\omega_p^{-2}\oplus \chi_p\lambda\omega_p^{-1}\oplus \chi_p\lambda\omega_p^{-2}\oplus \chi_p\lambda\omega_p^{-1}\oplus \chi_p\lambda\omega_p^{-1}\oplus \chi_p\lambda.
\end{equation}
By Theorem \ref{prop:eps} $i)$ and Proposition \ref{propeps:cyctwist}, we obtain 
\[
\epsilon(\sigma_{F,\chi, p}, \psi_p, \d x_p)=\lambda^8\omega_p^{-12}(p^{(n(\psi)\dim(\chi_p)+a(\chi_p))})\epsilon(\chi_p, \psi, \d x)^8=p^{12}\epsilon(\chi_p, \psi_p, \d x_p)^8, 
\]
since $a(\chi_p)=1$ and $\lambda$ is a quadratic character.
By Corollary \ref{coro:minusone}, we see that 
\[
\epsilon(\sigma_{F,\chi,p}, \psi_p, \d x_p)=p^{12}(p\chi_p(-1))^4=p^{16}.
\]
In conclusion, we deduce that $W(\sigma'_{F,\chi,p})=1$, which completes the proof.
\end{proof}

\begin{remark}
We proceed to extract the conductor $\cond(F,\chi)$ from the proof. When $q$ is distinct from $p$, we saw that $\sigma'_{F,\chi, q}$ is unramified, hence $a(\sigma'_{F,\chi, q})=0$. At the prime $p$, we have 
\begin{multline*}
a(\sigma'_{F,\chi,p}) = a(\sigma_{F,\chi,p})+\dim V^{I_p}/(V^{I_p}\cap \ker N_{F,\chi,p})= a(\sigma_{F,\chi,p}) \\
=a(\chi_p\lambda\omega_p^{-3})+3a(\chi_p\lambda\omega_p^{-2})+3a(\chi_p\lambda\omega_p^{-1})+a(\chi_p\lambda) 
    =8a(\chi_p)=8.
\end{multline*}
We conclude that
\[
\cond(F,\chi)=\prod_\ell \ell^{a(\sigma'_{F,\chi, p})}=p^8.
\]
\end{remark}

\begin{remark}\label{rem:end}
    By contrast, the same techniques applied to the untwisted case lead to the equalities $$\delta(\sigma'_{F, p})=-p^{10}\lambda^5(\Phi)=-p^{10}a_p(f_1)a_p(f_2)a_p(f_3)$$ and $\epsilon(\sigma_{F, p}, \psi_p, \d x_p)=1$. This in turn implies that $W(F)=a_p(f_1)a_p(f_2)a_p(f_3)$, in agreement with \cite{grosskudla}. Moreover, the conductor in this case is $\cond(F)=p^5$.
\end{remark}

\section*{Acknowledgements}
The author thanks Henri Darmon and Jan Vonk for helpful comments and suggestions, as well as the anonymous referee for valuable feedback.
Part of this work was done while the author was supported by a Scholarship for Outstanding PhD Candidates (ISM) at McGill University. The author is currently supported by an Edixhoven Post-Doctoral Fellowship at Leiden University.

\end{document}